\documentclass[12pt]{article}
\usepackage{amssymb,amsmath,enumerate, tikz, hyperref, amsthm, amsfonts,comment}
\usepackage[top=25mm, bottom=25mm, left=28mm, right=28mm]{geometry}
\newfont{\msbm}{msbm10 scaled\magstephalf}

\newtheorem{theorem}{Theorem}[section]

\newtheorem{proposition}[theorem]{Proposition}

\newtheorem{lemma}[theorem]{Lemma}

\newtheorem{claim}[theorem]{Claim}

\newtheorem{observation}[theorem]{Observation}

\newtheorem{conjecture}[theorem]{Conjecture}

\newcommand{\ex}{\mathrm{ex}}

\newcommand{\FF}{\mbox{\msbm{F}}}

\newcommand{\PP}{\mbox{\msbm P}}
\newcommand{\EE}{\mbox{\msbm E}}

\newtheorem{prob}{Problem}

\title{On the generalized Tur\'an number of complete bipartite graphs}
\author{Oliver Janzer\footnote{Institute of Mathematics, EPFL,  Lausanne, Switzerland, {\tt oliver.janzer@epfl.ch}} \and Sean Longbrake\footnote{Dept. of Mathematics, Emory University,  Atlanta, GA 30322, USA, {\tt sean.longbrake@emory.edu}} \and Liana Yepremyan\footnote{Dept. of Mathematics, Emory University,  Atlanta, GA 30322, USA, {\tt lyeprem@emory.edu}.  Research is supported by the National Science Foundation grant 2247013: Forbidden and Colored Subgraphs.}}
\date{}

\begin{document}

\maketitle

\begin{abstract}
    For graphs $F$ and $H$, the generalized Tur\'an number $\mathrm{ex}(n,F,H)$ denotes the maximum number of copies of $F$ in an $H$-free graph on $n$ vertices. We prove that if $s\in \{2,3\}$, $s< a\leq b$ and $t$ is sufficiently large, then $\mathrm{ex}(n,K_{a,b},K_{s,t})=\Theta(n^s)$. The $s=2$, $a=b=3$ case of this result answers a question of Spiro.

    Proving another conjecture of Spiro, we show that for every graph $F$ with at least one edge, there exist infinitely many real numbers $r$ such that $\mathrm{ex}(n,F,H)=\Theta(n^r)$ holds for some graph~$H$.
\end{abstract}

\section{Introduction}

The Tur\'an problem is one of the central topics in Graph Theory. It is concerned with estimating, for a graph $H$ and positive integer $n$, the maximum number of edges that an $n$-vertex, $H$-free graph can have. A natural generalization of this problem was introduced by Alon and Shikhelman \cite{alon2016many}. For graphs $F$ and $H$ and a positive integer $n$, they defined $\ex(n,F,H)$ to be the maximum possible number of copies of $F$ in an \emph{$H$-free} graph on $n$ vertices, that is, in an $n$-vertex graph which does not contain $H$ as a subgraph. This function became known as the \emph{generalized Tur\'an number} of $F$ and $H$. The name reflects the fact that for $F=K_2$, we have that $\ex(n,F,H)$ is just the usual Tur\'an number $\ex(n,H)$. Although the systematic study of generalized Tur\'an numbers was only initiated recently, many special cases had already been studied for decades. The first such result (in which $F$ is different from $K_2$) was obtained in 1949 by Zykov \cite{Zyk49}, who determined $\ex(n,K_t,K_r)$ for all $t<r$. The same result was proven later but independently by Erd\H os \cite{Erd62} and Roman~\cite{Rom76}.

While the generalized Tur\'an number of complete graphs has been understood for over 70 years, the corresponding problem for complete bipartite graphs has remained wide open. Given positive integers $a\leq b$ and $s\leq t$, we are interested in the growth of the function $\ex(n,K_{a,b},K_{s,t})$. We will focus on the case $s\geq 2$, as $s=1$ is much simpler. Clearly, if $K_{a,b}$ contains $K_{s,t}$ as a subgraph, then this function is identically zero. Hence, the question is only interesting if $a<s$ or $b<t$. 

The problem is quite well understood in the case $a\leq s$. When $a=b=1$, the question specialises to the Tur\'an problem for $K_{s,t}$, which is also known as the Zarankiewicz problem. The classical K\H ov\'ari--S\'os--Tur\'an theorem \cite{KST54} states that $\ex(n,K_{s,t})=O(n^{2-1/t})$. Matching lower bounds are known only when $t$ is sufficiently large compared to $s$ (see \cite{KRS96,ARS99,bukh2024extremal}). Alon and Shikhelman \cite{alon2016many} determined the order of magnitude of $\ex(n,K_{a,b},K_{s,t})$ in the case where $K_{a,b}$ is ``substantially smaller'' than $K_{s,t}$, and their results were extended in various directions in several papers \cite{ma2018some,bayer2019exploring,gerbner2022generalized}. In particular, Ma, Yuan and Zhang \cite{ma2018some} determined the order of magnitude of $\ex(n,K_{a,b},K_{s,t})$ whenever $t$ is sufficiently large compared to $a,b,s$ and either $a=s$ or $a\leq b< s$. The order of magnitude in the case $a<s\leq b$ is trivial for all values of $t$ (see \cite{gerbner2022generalized}). These results together cover all cases with $a\leq s$ and sufficiently large $t$. 

On the other hand, very little is known about the order of magnitude of $\ex(n,K_{a,b},K_{s,t})$ when $a>s$. Notice that for the question to be meaningful, we must have $s< a\leq b < t$. For this regime, we have the trivial upper bound $\ex(n,K_{a,b},K_{s,t})=O(n^s)$, first observed in \cite{gerbner2019generalized}. Indeed, if a graph $G$ is $K_{s,t}$-free, then any given $s$ vertices extend to at most $O_t(1)$ copies of $K_{a,b}$ in $G$. Prior to our work, a matching lower bound was not known for any $2\leq s<a\leq b<t$. The smallest instance of this problem concerns $\ex(n,K_{3,3},K_{2,t})$, and Spiro \cite{spiro} asked to determine the order of magnitude of this function for large $t$.

In this paper, we answer Spiro's question, and, more generally, determine the order of magnitude of $\ex(n,K_{a,b},K_{s,t})$ for $s=2$ and $s=3$, whenever $t$ is sufficiently large.

\begin{theorem} \label{thm:k3t-free}
    For any $s\in \{2,3\}$ and any $s< a\leq b$, there exists some $t_0$ such that for all $t\geq t_0$, we have
    $$\ex(n,K_{a,b},K_{s,t})=\Theta_{t,b}(n^s).$$
\end{theorem}

In fact, as we will see, the $s=2$ case of the above theorem follows from (a minor strengthening of) the $s=3$ case.

Our other main result concerns realizable exponents for generalized Tur\'an numbers. A number $r\in [1,2]$ is called a \emph{realizable exponent} if there exists a graph $H$ such that $\ex(n,H)=\Theta(n^r)$. The celebrated Rational Exponents Conjecture of Erd\H os and Simonovits \cite{erdHos1981combinatorial} asserts that every rational number $r\in [1,2]$ is a realizable exponent. Despite substantial recent progress \cite{bukh2018rational, JiangMaYepremyan, kang2021rational, CJL, JiangQiu1, JanzerSub, JiangQiu2, conlon2022rational}, this conjecture is still wide open.

The notion of a realizable exponent extends naturally to generalized Tur\'an numbers. Given a graph $F$, we call a number $r$ a \emph{realizable exponent for $F$} if there exists a graph $H$ for which $\ex(n,F,H)=\Theta(n^r)$. Observe that when $F=K_2$, we recover the usual notion of a realizable exponent. We note that the analogous notion for forbidden families (rather than a single forbidden graph) was investigated thoroughly in a recent paper of English and Spiro \cite{english2025rational}.

Spiro \cite{spiro} conjectured that there are infinitely many realizable exponents for any graph $F$. We prove this conjecture.

\begin{theorem} \label{thm:realizable exponents}
    For every graph $F$ with at least one edge, there are infinitely many real numbers $r$ such that there exists some graph $H$ with $\ex(n,F,H)=\Theta(n^r)$.
\end{theorem}

The rest of this paper is organized as follows. In Section \ref{sec:overview}, we give an overview of the proof of Theorem \ref{thm:k3t-free}. In Section \ref{sec:preliminaries}, we collect a few preliminary lemmas that we will use in our proofs. In Section \ref{sec:k3t-free}, we prove Theorem \ref{thm:k3t-free}. Finally, in Section \ref{sec:realizable exponents}, we prove Theorem \ref{thm:realizable exponents}.

\paragraph{Note added.} While we were finalizing this paper, independently from us and each other, Pohoata, Tidor and Yu \cite{PTY26} and Taranchuk \cite{taranchuk} obtained a similar result to the $s=2$ case of our Theorem~\ref{thm:k3t-free}. They show that one can take $t_0=b+1$ in this result. The approach of Pohoata, Tidor and Yu is similar to ours, whereas Taranchuk's proof is significantly different. Neither of these papers deal with $K_{s,t}$-free graphs for any $s\geq 3$. One of our main contributions is solving the problem for $K_{3,t}$-free graphs, which requires several new ideas beyond the $K_{2,t}$-free case, as will be discussed in the next section.

\section{Proof overview} \label{sec:overview}

Although the formal proof of the $s=2$ case of Theorem \ref{thm:k3t-free} will be deduced from the $s=3$ construction, it is helpful first to explain how one can directly construct a suitable $K_{2,t}$-free graph. We will then discuss the additional difficulties for extending this to $s=3$ and how we overcome them.

Our first goal is to define an $n$-vertex $K_{2,t}$-free graph $G$ which contains $\Theta(n^2)$ copies of $K_{a,b}$, where $t$ is large enough compared to $b$. For simplicity, let us assume that $a=b=3$ as the general case is nearly identical. It is easy to see that in any $n$-vertex $K_{2,t}$-free graph, the number of copies of $K_{3,3}$ minus an edge is at most $O_t(n^2)$. Hence, in order to have $\Theta_t(n^2)$ copies of $K_{3,3}$ in our graph $G$, we must have the property that a positive constant proportion of the copies of $K_{3,3}-e$ extend to a $K_{3,3}$. This is a very atypical property for a general sparse graph, but it does hold for every point-plane incidence graph in $3$ dimensions, as exploited in recent work of Milojevi\'c, Sudakov and Tomon \cite{milojevic2024incidence}. This suggests that we should take $G$ to be a point-plane incidence graph in $\mathbb{F}_q^3$. Let $\mathcal{P}$ be the set of points and let $\mathcal{H}$ be the set of planes in this incidence graph. If each line in $\mathbb{F}_q^3$ contains less than $t$ points of $\mathcal{P}$ and each line is contained in less than $t$ planes of $\mathcal{H}$, then the associated incidence graph $G$ is $K_{2,t}$-free (this is because any two points determine a line, and any two non-parallel planes intersect in a line). To have many copies of $K_{3,3}$ in $G$, we want that there are many lines in $\mathbb{F}_q^3$ simultaneously containing (at least) three points in $\mathcal{P}$ and contained in (at least) three planes in $\mathcal{H}$. That is, there should be many lines with three points on them, but none with $t$ points on them. This suggests that $\mathcal{P}$ should contain around $q^2$ points, in which case a typical line in $\mathbb{F}_q^3$ contains constantly many points from $\mathcal{P}$. To ensure that no line contains more than $t$ points of $\mathcal{P}$, we can use the random algebraic method of Bukh \cite{bukh2015random}, similarly to how it is done in a paper of Sudakov and Tomon \cite{SudakovTomon} on subspace evasive sets. To summarise, we find a set of points $\mathcal{P}$ and a set of planes $\mathcal{H}$ in $\mathbb{F}_q^3$ with the following properties:
\begin{itemize}
    \item $|\mathcal{P}|\approx q^2$,
    \item $|\mathcal{H}|\approx q^2$,
    \item no line in $\mathbb{F}_q^3$ contains at least $t$ points from $\mathcal{P}$,
    \item no line in $\mathbb{F}_q^3$ is contained in at least $t$ planes from $\mathcal{H}$,
    \item a positive constant proportion of the lines in $\mathbb{F}_q^3$ contain at least three points of $\mathcal{P}$ and are, simultaneously, contained in at least three planes in $\mathcal{H}$.
\end{itemize}
The associated point-plane incidence graph is $K_{2,t}$-free, has around $2q^2$ vertices, and has $\Omega(q^4)$ copies of $K_{3,3}$ since the number of lines in $\mathbb{F}_q^3$ is $\Theta(q^4)$. Thus, it gives a suitable construction for $\ex(n,K_{3,3},K_{2,t})=\Omega(n^2)$.

We will now explain what additional difficulties arise when we try to extend this approach to $K_{3,t}$-free graphs. Let us say that we would like to maximize the number of copies of $K_{4,4}$ in an $n$-vertex $K_{3,t}$-free graph. With a very similar argument as above, one can prove that there exist a set of points $\mathcal{P}$ and a set of affine hyperplanes $\mathcal{H}$ in $\mathbb{F}_q^5$ with the following properties:
\begin{itemize}
    \item $|\mathcal{P}|\approx q^3$,
    \item $|\mathcal{H}|\approx q^3$,
    \item no plane in $\mathbb{F}_q^5$ contains at least $t$ points from $\mathcal{P}$,
    \item no plane in $\mathbb{F}_q^5$ is contained in at least $t$ hyperplanes from $\mathcal{H}$,
    \item a positive constant proportion of the planes in $\mathbb{F}_q^5$ contain at least four points of $\mathcal{P}$ and are, simultaneously, contained in at least four hyperplanes in $\mathcal{H}$.
\end{itemize}

Let $G$ be the associated point-hyperplane incidence graph. It is easy to see that the last property implies that the number of copies of $K_{4,4}$  in $G$ is at least $\Omega(q^{9})$ (assuming that no four points in $\mathcal{P}$ are collinear), as the number of planes in $\mathbb{F}_q^5$ is $\Theta(q^9)$. Unfortunately, however, $G$ is not guaranteed to be $K_{3,t}$-free. Indeed, if there exist three collinear points in $\mathcal{P}$, then they determine a line, rather than a plane, so the fourth property does not rule out that they are simultaneously contained in at least $t$ hyperplanes. In fact, since the number of hyperplanes in $\mathcal{H}$ is around $q^3$, a typical line in $\FF_q^5$ is contained in around $q$ hyperplanes in $\mathcal{H}$. (This is because in $\FF_q^5$, there are roughly $q^8$ lines and each hyperplane contains roughly $q^6$ lines.) To rule out copies of $K_{3,t}$ arising this way, we would like to require that in addition to the above five properties, the following two are satisfied as well:
\begin{itemize}
    \item no three points in $\mathcal{P}$ are collinear, \item no three hyperplanes in $\mathcal{H}$ have a $3$-dimensional intersection.
\end{itemize}

If we can find $\mathcal{P}$ and $\mathcal{H}$ with the seven properties described, then we have a suitable construction. While it may be true that such $\mathcal{P}$ and $\mathcal{H}$ exist, we were unable to show this and instead weakened the requirements as follows. Notice that in order to avoid copies of $K_{3,t}$ formed by $t$ points and three hyperplanes, we do not need to require that \emph{every} plane contains fewer than $t$ points from $\mathcal{P}$, only that if a plane contains at least $t$ points, then it should be contained in fewer than three hyperplanes from $\mathcal{H}$. Hence, if we have only a few ``exceptional'' planes which contain many points from $\mathcal{P}$, we can simply delete all hyperplanes from $\mathcal{H}$ that contain one of these planes. Observe that in $\FF_q^5$ a typical plane is contained in constantly many hyperplanes from $\mathcal{H}$. (Indeed, this is because there are roughly $q^9$ planes in $\FF_q^5$, and any hyperplane contains about $q^6$ planes.)  Hence, this deletion process removes roughly as many elements of $\mathcal{H}$ as the number of exceptional planes. Thus, we want to ensure that the number of exceptional planes is much less than $q^3$. To summarize, after deletion we still expect around $q^3$ hyperplanes surviving (so hopefully the incidence graph $G$ still has $\Omega(q^9)$ copies of $K_{4,4}$), but now $G$ does not contain a copy of $K_{3,t}$ with $t$ points and three hyperplanes. Similarly, deleting all points in $\mathcal{P}$ contained in any plane that is contained in at least $t$ hyperplanes from $\mathcal{H}$, we destroy all copies of $K_{3,t}$ with three points and $t$ hyperplanes, so we end up with a $K_{3,t}$-free construction.

Thus, our task now is to find a point set $\mathcal{P}$ in $\mathbb{F}_q^5$ such that

\begin{itemize}
    \item $|\mathcal{P}|\approx q^3$,
    \item no three points from $\mathcal{P}$ are collinear,
    \item there are ``few'' planes in $\mathbb{F}_q^5$ containing at least $t$ points from $\mathcal{P}$,
    \item a positive constant proportion of the planes in $\mathbb{F}_q^5$ contain at least four points of $\mathcal{P}$.
\end{itemize}

Once we have such a point set, it is not hard to construct an analogous set $\mathcal{H}$ of hyperplanes by a simple duality argument. We will prove that

$$\mathcal{P}=\{(x,y,x^2-\alpha y^2,z,z^2+P(x,y)): x,y,z\in \mathbb{F}_q\},$$
where $\alpha$ is a quadratic nonresidue in $\mathbb{F}_q$ and $P$ is a random polynomial of bounded degree, is a suitable construction with positive probability. The proof is given in Section \ref{sec:k3t-free}; here we just remark that the exceptional planes (planes containing at least $t$ points of $\mathcal{P}$) are those of the form $Q_{x,y}:=\{(x,y,x^2-\alpha y^2,u,v):u,v\in \mathbb{F}_q\}$ for any fixed $x,y\in \mathbb{F}_q$. It is clear that these planes contain $q$ elements of $\mathcal{P}$; on the other hand, we will prove that (with high probability) these are the only exceptional planes, and therefore there are at most $q^2$ exceptional planes, as desired.

\section{Preliminaries} \label{sec:preliminaries}

We will use the following instance of a lemma from Bukh and Conlon~\cite{bukh2018rational} which says that a variety over $\mathbb{F}_q$ of bounded complexity either has size $O(1)$ or at least $\Omega(q)$ when $q$ is sufficiently large. 

\begin{lemma}[c.f Lemma 2.7 in \cite{bukh2018rational}]\label{boundinglemma}
    For every $d$, there exists a constant $t_0$ such that the following holds for every field $\FF_q$ with $q$ sufficiently large in terms of $d$. Let $m, n \leq d$. Suppose we have polynomials $p_1(x_1, \dots ,x_n)$, $p_2(x_1, \dots ,x_n)$, $\dots ,p_m(x_1, \dots ,x_n)$ over $\mathbb{F}_q$, each of degree at most $d$. Then the size of the variety $V = \{(z_1, \dots ,z_n)\in \mathbb{F}_q^n : \forall i \in [m], p_i(z_1, \dots ,z_n) = 0 \}$ is either less than $t_0$ or at least $q/2$. 
\end{lemma}

 We will use the following version of a lemma of Conlon \cite{conlon2019graphs} in our proof, which roughly states that the events that a random polynomial takes on given values are uniform and independent for a bounded number of points. Conlon's lemma is stated only for $c_i = 0$, but it is relatively straightforward to extend it to the general $c_i \in \FF_q$ case. 
\begin{lemma}[c.f. Lemma 2 in \cite{conlon2019graphs}]\label{conlon-nonzero}
Suppose that $q > \binom{m}{2}$ and $d \geq m - 1$. Then, if $f$ is a uniform random polynomial in $t$ variables over $\FF_q$ of degree at most $d$ and $x_1, \dots, x_m$ are $m$ distinct points in $\FF_q^t$ and $c_1, \dots, c_m$ are some elements of $\FF_q$ then 
$$\PP(f(x_i) = c_i \text{ for all } i \in [m]) = \frac{1}{q^m}. $$
\end{lemma}

We make use of the following weaker version of a proposition of Hartshorne \cite{hartshorne},
which will allow us to embed our point set into projective space while preserving its key properties. This will be convenient for our construction of the hyperplane set, as we will employ duality. 
\begin{proposition}[c.f. Proposition I.2.2 in \cite{hartshorne}]\label{affinetoprojective}
    There is an embedding $\varphi:\mathbb{F}_q^n \rightarrow {\rm{PG}}(n, q)$ where a set $A\subseteq \mathbb{F}_q^n$ is a subset of 
    a $k$-dimensional affine subspace of $\FF_q^n$ if and only if $\varphi(A)$ is a subset of a $k$-dimensional projective subspace of ${\rm{PG}}(n,q)$. 
\end{proposition}
We also often make use of the following well known fact. Here and throughout the paper, we use $o(1)$ to represent a function that goes to zero as $q$ goes to infinity. 
\begin{proposition} \label{prop:subspace count}
    The number of $k$-dimensional affine (projective) subspaces in affine (projective) $n$-space over $\mathbb{F}_q$ is $(1 +o(1))q^{(k + 1)(n - k)}$.
\end{proposition}

In our lower bound for Theorem~\ref{thm:realizable exponents}, we will use the following proposition of Ma, Yuan and Zhang \cite{ma2018some}, obtained using the random algebraic method of Bukh \cite{bukh2015random}. 

\begin{proposition}[Theorem 1.4 in \cite{ma2018some}] \label{prop:realizable lower bound}
Let $F$ be a graph. Suppose that $t$ is sufficiently large in terms of $s$ and $F$. Then 
$$\ex(n, F, K_{s, t}) = \Omega(n^{v(F) - \frac{e(F)}{s}}).$$
\end{proposition}

\section{A construction of $K_{3,t}$-free graphs with many copies of $K_{a,b}$} \label{sec:k3t-free}

In this section, we prove Theorem \ref{thm:k3t-free}. In order to be able to deduce the $s=2$ case from the $s=3$ case of this result, we strengthen the conclusion very slightly by proving the following theorem. Note that the extra condition compared to Theorem \ref{thm:k3t-free} is the maximum degree bound.

\begin{theorem} \label{thm:k3t-free with max degree}
    For any $3<a\leq b$, there exists some $t_0$ such that for all sufficiently large~$n$, there is a bipartite $K_{3,t_0}$-free $n$-vertex graph with $\Omega_b(n^3)$ copies of $K_{a,b}$ and maximum degree $O_b(n^{2/3})$.
\end{theorem}

Before we turn to the proof of Theorem \ref{thm:k3t-free with max degree}, we show how to deduce Theorem \ref{thm:k3t-free} from it.

\begin{proof}[Proof of Theorem \ref{thm:k3t-free}]
    We start with the case $s=3$. Choose $t_0$ according to Theorem~\ref{thm:k3t-free with max degree}. Now for $t\geq t_0$ the lower bound $\ex(n,K_{a,b},K_{s,t})=\Omega_b(n^3)$ follows immediately from Theorem \ref{thm:k3t-free with max degree}. For the upper bound, note that in a $K_{3, t}$-free $n$-vertex graph, the number of copies of $K_{a, b}$ is less than $n^3 \cdot t^{a + b - 3} $. Indeed, for any triple of vertices, the number of common neighbors is at most $t-1$. It follows that $\ex(n,K_{a,b},K_{3,t})=O_{b,t}(n^3)$.

    We now turn to the $s=2$ case. By Theorem \ref{thm:k3t-free with max degree} (applied with the parameters $a+1,b+1,n^{3/2}$ in place of $a,b,n$), there exists some $t_0$ such that for all sufficiently large $n$, there is a bipartite $K_{3,t_0+1}$-free $n^{3/2}$-vertex graph $G_n$ with $\Omega_b(n^{9/2})$ copies of $K_{a+1,b+1}$ and maximum degree $O_b(n)$. We claim that $G_n$ contains a $K_{2,t_0}$-free subgraph $H_n$ on $O_b(n)$ vertices with $\Omega_b(n^2)$ copies of $K_{a,b}$. Indeed, choose the edge $uv$ in $G_n$ that is contained in the most copies of $K_{a+1,b+1}$. Since $G_n$ has at most $|V(G_n)|\Delta(G_n)\leq O_b(n^{5/2})$ edges, there exist $\Omega_b(n^2)$ copies of $K_{a+1,b+1}$ in $G_n$ containing $uv$. Now let $H_n=G_n[(N(u)\setminus \{v\})\cup (N(v)\setminus \{u\})]$. Since $G_n$ has maximum degree $O_b(n)$, $H_n$ has $O_b(n)$ vertices. Note that $H_n$ is $K_{2,t_0}$-free, as any $K_{2,t_0}$ in $H_n$ could be extended to a $K_{3,t_0+1}$ in $G_n$, using the edge $uv$. Finally, each copy of $K_{a+1,b+1}$ in $G_n$ containing $uv$ yields a copy of $K_{a,b}$ in $H_n$. Hence, $H_n$ contains $\Omega_b(n^2)$ copies of $K_{a,b}$. The existence of $H_n$ therefore implies (by adding isolated vertices if $H_n$ has fewer than $n$ vertices, and taking a random $n$-vertex induced subgraph if $H_n$ has more than $n$ vertices) that $\ex(n,K_{a,b},K_{2,t_0})=\Omega_b(n^2)$, proving the lower bound. For the upper bound, note that in a $K_{2, t}$-free $n$-vertex graph, the number of copies of $K_{a, b}$ is less than $n^2 \cdot t^{a + b - 2} $. Indeed, for any pair of vertices, the number of common neighbors is at most $t-1$. It follows that $\ex(n,K_{a,b},K_{2,t})=O_{b,t}(n^2)$.
\end{proof}

It remains to prove Theorem \ref{thm:k3t-free with max degree}.
For every sufficiently large odd prime power $q$, we will build a $K_{3,t_0}$-free bipartite graph on $\Theta(q^3)$ vertices of maximum degree $O(q^2)$ with $\Theta(q^9)$ many copies of $K_{a,b}$.

We will use a point set as given by the following lemma. 

\begin{lemma}\label{goodsset}
    For every integer $b\geq 3$, there exists $\delta > 0$ and $t_0$ such that the following holds for every sufficiently large odd prime power $q$. There exists a set of points $S \subseteq \FF_q^5$ such that:
    \begin{enumerate}
        \item $|S| = q^3$, 
        \item no three points of $S$ are collinear,
        \item there are $q^2$ planes in $ \FF_q^5$ which intersect $S$ in exactly $q$ points; we call these planes  \emph{dangerous} for $S$,
        \item if a plane in $ \FF_q^5$ is not dangerous, then it intersects $S$ in fewer than $t_0$ points, 
        \item there are at least $\delta q^9$ sets of $b$ points in $S$ that are coplanar and the plane they lie in is not dangerous for $S$.  
    \end{enumerate}
    \end{lemma}

We now describe how to prove Theorem~\ref{thm:k3t-free with max degree} given Lemma~\ref{goodsset}. 

\begin{proof}[Proof of Theorem~\ref{thm:k3t-free with max degree}]

Fix a set $S$ satisfying the conclusions of Lemma~\ref{goodsset}. We will use this set to construct a $K_{3, t_0}$-free graph with many $K_{a, b}$'s. 

We pass from $S \subseteq \FF_q^5$ to $\varphi(S)$ in ${\rm{PG}}(5, q)$ by using the map $\varphi$ from Proposition~\ref{affinetoprojective}.  Note that a set of points in $S$ is collinear (coplanar) if and only if the set of corresponding points in $\varphi(S)$ is collinear (coplanar). We say a plane $Q$ in $ {\rm{PG}}(5, q)$ is \emph{dangerous} for $\varphi(S)$ if there are exactly $q$ points of $\varphi(S)$ that lie on the plane $Q$. By Proposition~\ref{affinetoprojective}, there are $q^2$ dangerous planes for $\varphi(S)$, and every other plane contains strictly less than $t_0$ points of $\varphi(S)$.

Given a subspace $A \subseteq {\rm{PG}}(5, q)$, recall that the orthogonal complement of $A$ is defined as $A^\perp = \{y\in {\rm{PG}}(5,q) : \langle y, a \rangle = 0 \text{ for all } a \in A\}$. Note that if $A$ is a point, then $A^\perp$ is a hyperplane; if $A$ is a line, then $A^{\perp}$ is a $3$-dimensional subspace; and if $A$ is a plane, then $A^{\perp}$ is a plane. Furthermore, observe that containment is reversed when taking orthogonal complements.

Take a random isomorphism $\textbf{T}$ of the space ${\rm{PG}}(5, q)$. Let $\bold{H}$ be the set of hyperplanes $H$ such that for some point $P \in \varphi(S)$, $H = \textbf{T}(P)^{\perp}$. By the conditions on $\varphi(S)$, we have the following conditions on $\bold{H}$.
\begin{enumerate}
    \item $|\bold{H}|=q^3$,
    \item no three hyperplanes in $\bold{H}$ have a three-dimensional intersection,
    \item there are $q^2$ planes in ${\rm{PG}}(5,q)$ which are contained in $q$ hyperplanes from $\bold{H}$; we call these planes \emph{dangerous} for $\bold{H}$, 
    \item if a plane in ${\rm{PG}}(5,q)$ is not dangerous for $\bold{H}$, then it is contained in fewer than $t_0$ hyperplanes of $\bold{H}$, 
    \item there are at least $\delta q^9$ sets of $b$ hyperplanes in $\bold{H}$ whose intersection is a plane which is not dangerous for $\bold{H}$.
\end{enumerate}  

Let $\bold{H}'$ be the subset of $\bold{H}$ formed by removing from $\bold{H}$ all hyperplanes containing a dangerous plane with respect to $\varphi(S)$. Similarly, let $\bold{S}'$ be the subset of $\varphi(S)$ formed by removing from $\varphi(S)$ all points contained in a dangerous plane with respect to $\bold{H}$. The next claim will essentially finish the proof since $|\bold{S}'|+|\bold{H}'|\leq 2q^3$.

\begin{claim} \label{claim:random incidence graph}
    Let $\bold{G}$ be the bipartite graph on vertex set $\bold{S'}\cup \bold{H'}$, with an edge between $P\in \bold{S}'$ and $H\in \bold{H}'$ if $P \in H$. Then,
    \begin{enumerate}
        \item[(i)] $\textbf{G}$ is $K_{3, t_0}$-free,
        \item[(ii)] $\Delta(\bold{G}) = O( q^2)$, 
        \item[(iii)] the expected number of $K_{a, b}$'s in $\textbf{G}$ is  $\Omega(q^9).$
    \end{enumerate}
\end{claim}
\begin{proof}[Proof of (i):]
    Suppose there was some $K_{3, t_0}$ in $\textbf{G}$. Assume that it is formed by three points in $\bold{S}'$ and $t_0$ hyperplanes in $\bold{H}'$ (the other case follows from a dual argument). By our choice of $\bold{S}'$, the three points are not collinear, so they define a plane. This plane is contained in at least $t_0$ hyperplanes in $\bold{H}'$ and hence also in at least $t_0$ hyperplanes in $\bold{H}$, therefore this plane is dangerous for $\bold{H}$. However, we had deleted all points of $\varphi(S)$ contained in a dangerous plane for $\bold{H}$, which is a contradiction. Thus, $\textbf{G}$ is $K_{3, t_0}$-free.  \end{proof}
    \begin{proof}[Proof of (ii):]Let $H\in \bold{H}'$. Let $\bold{N}(H)$ be its neighborhood in $\bold{G}$. By Proposition~\ref{prop:subspace count}, $H$ contains $(1+o(1))q^6$ planes. Note that any three points in $\bold{N}(H)$ determine a plane contained in the hyperplane $H$, as $\bold{N}(H) \subseteq \bold{S}'\subseteq \varphi(S)$ contains no three collinear points. On the other hand, every plane contained in $H$ contains fewer than $t_0$ points from $\bold{N}(H)$ (as otherwise it would be dangerous for $\varphi(S)$ in which case we would have deleted $H$). Combining these inequalities, we have ${|\bold{N}(H)| \choose 3}/t_0^3\leq (1+o(1))q^6$, which implies  $|\bold{N}(H)|<8t_0q^2$.

Take now some point $P \in \bold{S}'$, and let $\bold{N}(P)$ be its neighborhood in $\bold{G}$. By Proposition~\ref{prop:subspace count}, $P$ is contained in $(1+o(1))q^6$ planes. Note that any three hyperplanes in $\bold{N}(P)$ intersect in a plane containing the point $P$, as $\bold{N}(P) \subseteq \bold{H}'\subseteq \bold{H}$ contains no three hyperplanes with a three-dimensional intersection. On the other hand, every plane containing $P$ is contained in fewer than $t_0$ hyperplanes from $\bold{N}(P)$ (as otherwise it would be dangerous for $\bold{H}$ in which case we would have deleted $P$). Combining these inequalities, we have ${|\bold{N}(P)| \choose 3}/t_0^3\leq (1+o(1))q^6$, which implies  $|\bold{N}(P)|<8t_0q^2$.
    \end{proof}
    
\begin{proof}[Proof of (iii):] We now seek to count the number of  $K_{a, b}$'s in $\textbf{G}$. We begin by counting the number of $K_{b, b}$'s in $\textbf{G}$; the desired lower bound for $K_{a,b}$'s will then follow easily. 

We call a set of $b$ coplanar points from $\varphi(S)$ \emph{suitable} if the plane they determine is not dangerous for $\varphi(S)$. Similarly, we call a set of $b$ hyperplanes from $\bold{H}$ that intersect in a plane \emph{suitable} if the plane they intersect in is not dangerous for $\bold{H}$. By the choice of $S$, there are at least $\delta q^9$ suitable sets in $\varphi(S)$.

For a suitable set of points $B$, let  $\bold{H}(B) = \{ \textbf{T}(P)^{\perp} : P \in B\}$. Note that a set of points $B$ is suitable if and only if the set of hyperplanes $\bold{H}(B)$ is suitable. 
Fix two suitable sets $A$ and $B$ of $b$ coplanar points from $\varphi(S)$. 

Let $\Pi_A$ be the plane containing the $b$ points of $A$ and let $\Pi_B$ be the plane containing the $b$ points of $B$. 

We claim that the intersection of the hyperplanes in $\bold{H}(B)$ is $\textbf{T}(\Pi_B)^{\perp}$. Indeed for every point $P \in B$, $\bold{T}(P) \subseteq \bold{T}(\Pi_B)$, and therefore $\bold{T}(P)^{\perp} \supseteq \bold{T}(\Pi_B)^{\perp}$. As $\bold{H}$ contains no three hyperplanes meeting in a three-dimensional subspace, the intersection of the hyperplanes in $\bold{H}(B)$ is indeed exactly $\bold{T}(\Pi_B)^{\perp}$. 

Now, by Proposition \ref{prop:subspace count}, the probability that $\Pi_A = \textbf{T}(\Pi_B)^{\perp}$ is $(1 +o(1))q^{-9}$.  Summing up over all choices of a suitable $A$ and $B$, we have a lower bound of $(1+o(1))\delta^2 q^9$ for the expected number of $K_{b, b}$'s formed by picking a suitable set from $\varphi(S)$ and a suitable set from $\bold{H}$. Call such $K_{b, b}$'s suitable. 
Unfortunately, some suitable $K_{b, b}$'s were destroyed when we removed hyperplanes which contain a dangerous plane for $\varphi(S)$ and when we removed points which were contained in a dangerous plane for $\bold{H}$. We now will prove an upper bound on the expected number of destroyed $K_{b, b}$'s. 

Note that each point of $\varphi(S)$ (respectively, each hyperplane of $\bold{H}$) is in at most $t_0^{2b - 3}q^6$ suitable $K_{b, b}$'s. Indeed, there are at most $q^6$ choices for another two points from $\varphi(S)$. Since no three points of $\varphi(S)$ are collinear, these three points define some plane in $\textrm{PG}(5, q)$. Assuming that this plane contains at most $t_0$ points from $\varphi(S)$ and is contained in at most $t_0$ hyperplanes from $\bold{H}$ (as otherwise any formed $K_{b, b}$'s are not suitable), there are at most $t_0^{ b - 3}$ choices for the remaining points of $\varphi(S)$ and at most $t_0^b$ choices for the remaining points of $\bold{H}$.

Observe that the probability that a given point in $\varphi(S)$ was deleted is at most $(1+o(1))q^{-1}$. Indeed, each point is in $(1 + o(1))q^{6}$ planes, and the probability that a specific one of them is dangerous for $\bold{H}$ is $(1 + o(1))q^{-7}$. Thus, by the union bound, the probability that one of them is dangerous is at most $(1+o(1))q^{-1}$.
Similarly, for any $P\in \varphi(S)$, the probability that the hyperplane $\bold{T}(P)^{\perp}$ was deleted from $\bold{H}$ is at most $(1+o(1))q^{-1}$. Indeed, $\bold{T}(P)^{\perp}$ is a uniformly random hyperplane. We delete it from $\bold{H}$ precisely if it contains one of the $q^2$ dangerous planes for $\varphi(S)$. Since each plane is contained in $(1+o(1))q^2$ hyperplanes and there are $(1+o(1))q^5$ hyperplanes in total, it follows that the proportion of hyperplanes containing a dangerous plane for $\varphi(S)$ is at most $(1+o(1))q^{-1}$. Hence, $\bold{T}(P)^{\perp}$ gets deleted from $\bold{H}$ with probability at most $(1+o(1))q^{-1}$.

Thus, by the previous two paragraphs, the expected number of suitable $K_{b, b}$'s which were removed is at most $(2 + o(1))t_0^{2b - 3}q^8$. Thus, the expected number of $K_{b, b}$'s in $\textbf{G}$ is at least $\Omega(q^9)$. 

 Each $K_{a, b}$ in $\bold{G}$ can be extended to at most $t_0^{b - a}$ many $K_{b, b}$'s as $\bold{G}$ is $K_{3, t_0}$-free. Thus, the expected number of $K_{a, b}$'s in $\textbf{G}$ is at least $\Omega(q^9)$.
\end{proof}

Now if $n$ is sufficiently large, let $q$ be the largest prime with $2q^3\leq n$. By Claim \ref{claim:random incidence graph}, there exists a bipartite $K_{3,t_0}$-free graph $G_0$ with at most $2q^3$ vertices, $\Delta(G_0)=O(q^2)$ and $\Omega(q^9)$ copies of $K_{a,b}$. Adding isolated vertices if necessary, and noting that $q=\Theta(n^{1/3})$, we obtain a suitable $n$-vertex graph.
\end{proof}

The rest of this section is devoted to the proof of Lemma~\ref{goodsset}. 

\subsection{Proof of Lemma~\ref{goodsset}}
Fix a nonsquare $\alpha\in \FF_q$, and let $d = b + 9$. Let $t_0$ be the constant returned by Lemma~\ref{boundinglemma} for the integer $d$.

We will take
$$\bold{S} = \{(x, y, x^2 - \alpha y^2, z, z^2 + \textbf{P}(x, y)): x,y,z\in \mathbb{F}_q\}\subset \mathbb{F}_q^5,$$
where $\textbf{P}(x,y)$ is a random polynomial of degree at most $d$.

First we show that such a set does not contain three collinear points (for any choice of $\textbf{P}$), making note of the following two observations. 
\begin{observation}\label{frontprojectionsnotcollinear}
    The set $\{(x, y, x^2 - \alpha y^2):x, y \in \FF_q\}\subset \mathbb{F}_q^3$ does not contain three collinear points. 
\end{observation}

\begin{observation}\label{backprojectionsnotcollinear}
    For all $C \in \FF_q$, the set $\{(z, z^2 + C) : z \in \FF_q\}\subset \mathbb{F}_q^2$ does not contain three collinear points. 
\end{observation}

\begin{lemma} \label{lem:no collinear triple}
Let $\bold{S} = \{(x, y, x^2 - \alpha y^2, z, z^2 + \bold{P}(x, y)): x,y,z\in \mathbb{F}_q\}\subset \mathbb{F}_q^5$, with $\bold{P}(X,Y)$ some polynomial over $\mathbb{F}_q$. Then, $\bold{S}$ does not contain three collinear points. 
\end{lemma}
\begin{proof}
    Suppose on the contrary that there are three collinear points $P^1, P^2, P^3$ in $\bold{S}$. Then their projections on the first three coordinates are collinear. By Observation \ref{frontprojectionsnotcollinear}, at least two of the three projections must coincide. But recall that $P^1, P^2, P^3$ are collinear and distinct, so then all three projections must be the same. 

    Fix the $x, y$ such that these three projections are $(x, y, x^2 - \alpha y^2)$. Then, we have that the projections on the last two coordinates are three distinct collinear points on the parabola $\{(z, z^2 + \textbf{P}(x, y)): z\in \mathbb{F}_q\}$. This contradicts Observation \ref{backprojectionsnotcollinear}. 
\end{proof}

We now seek to classify how planes intersect the set $\bold{S}$.

Fix some plane $H \subseteq \FF_q^5$, and let $\pi(H)$ denote the projection of $H$ onto the first three coordinates. If the dimension of $\pi(H)$ is zero, we call such a plane \textit{risky}. If the dimension of $\pi(H)$ is one, we call such a plane \textit{tolerable}. If the dimension of $\pi(H)$ is two, we call such a plane \textit{useful}. 

\begin{lemma}\label{lemma-characterization}
For $\bold{S} = \{(x, y, x^2 - \alpha y^2, z, z^2 + \bold{P}(x, y)): x,y,z\in \mathbb{F}_q\}\subset \mathbb{F}_q^5$, with $\bold{P}(x, y)$ some uniform random polynomial of degree at most $d$, 
\begin{enumerate}
    \item[P1.] There are $q^2$ risky planes which intersect $\bold{S}$ and each one intersects $\bold{S}$ in exactly $q$ points. 
    \item[P2.] Each tolerable plane intersects $\bold{S}$ in at most $4$ points. 
    \item[P3.] The probability that there is some useful plane which intersects $\bold{S}$ in at least $t_0$ points is at most $O(q^{8 - d})$. 
\end{enumerate}
    
\end{lemma}

\begin{proof}[Proof of P1. ]
    If the dimension of $\pi(H)$ is zero, then we have that it contains a single point $(x, y, w)$. If $\bold{S} \cap H \neq \emptyset$, then we have $w =  x^2 - \alpha y^2$ for some $x, y \in \FF_q$. We have then that $H = \{(x, y, x^2 - \alpha y^2, u, v)$ for all $u, v \in \FF_q\}$. Thus, $|H \cap \bold{S}| = q$. Now, given that there are $q^2$ choices for $x, y$, there are $q^2$ choices for such a plane $H$.
\end{proof}

\begin{proof}[Proof of P2.]
     Suppose now that $\dim(\pi(H)) = 1$. Thus, $\pi(H \cap \textbf{S})$ is a set of collinear points, and so by Observation~\ref{frontprojectionsnotcollinear}, there are at most two of them. Fix one such projection $(x, y, x^2 - \alpha y^2)$. Now, since the $\dim(\pi(H)) = 1$, we have that the set of $(z, v)$ such that $(x, y, x^2 - \alpha y^2, z, v) \in H$ forms a line in $\FF_q^2$. Then by Observation~\ref{backprojectionsnotcollinear}, we have that $|H\cap \textbf{S} \cap \pi^{-1}(x, y, x^2 - \alpha y^2)| \leq 2$. Thus, $|H \cap \bold{S}| \leq 4$.  
\end{proof}

\begin{proof}[Proof of P3.] Let us examine some $H$ such that $\dim(\pi(H)) = 2$. Observe that $\pi(H)=\{(x,y,w)\in \mathbb{F}_q^3: \ell_H(x,y,w)=0\}$ for some degree $1$ polynomial $\ell_H$. Since $\dim(\pi(H)) = \dim(H) = 2$, for each $(x, y, w) \in \pi(H)$, there is a unique point in $H$ with first three coordinates $(x,y,w)$. Let $g_{4, H}(x, y, w), g_{5, H}(x, y, w)$ be the fourth and fifth coordinates of this unique point in $H$. Observe that $g_{4, H}$ and $g_{5, H}$ are polynomials of degree at most $1$. For simplicity of notation, let $g_H(x, y) = g_{5, H}(x, y, x^2 - \alpha y^2) - g_{4, H}(x, y, x^2 - \alpha y^2)^2$. 
Then, we have that \begin{equation*}
    |H \cap \bold{S}| = |\{(x, y) : \ell_H(x, y, x^2 - \alpha y^2) = 0 \text{ and } \textbf{P}(x, y) = g_H(x, y)\}|.
\end{equation*}

Note that for a given $H$, the number of $x, y$ such that $\ell_H(x, y, x^2 - \alpha y^2) = 0$ is no more than $2q$, as $\ell_H(x, y, x^2 - \alpha y^2)$ is a nonzero polynomial of degree two in two variables.

We now seek to show that the probability that $|H \cap \bold{S}| > t_0$ is $O(q^{- d - 1})$. To prove this, we will use a bound on the $(d + 1)$th moment of $|H \cap \bold{S}|$. 
\begin{claim}
    $$\EE[|H \cap \textbf{S}|^{d+1}] \leq (d + 1)^{d + 2}2^{ d+ 1}.$$
\end{claim}
\begin{proof}
    We use that, assuming that $k\leq d+1$ and $(x_1, y_1), \dots ,(x_k, y_k)$ are distinct points in $\FF_q^2$, $\PP(\forall i \in [k], \textbf{P}(x_i, y_i) = g_H(x_i, y_i)) =  \frac{1}{q^k}$ by Lemma~\ref{conlon-nonzero}. Let $\Gamma_H = \{ (x, y) \in \FF_q^2: \ell_H(x, y, x^2 - \alpha y^2) = 0\}$. Since the number of surjections $[d +1] \rightarrow [k]$ is at most $(d+1)^{d+1}$, we have

\begin{align*}
    \EE[|H \cap \bold{S}|^{d + 1}] &= \sum_{(u_1, v_1), \dots ,(u_{d + 1}, v_{d + 1}) \in \Gamma_H}\PP(\forall i \in [d + 1], \textbf{P}(u_i, v_i) = g_H(u_i, v_i) )\\
    &\leq \sum_{k = 1}^{ d + 1} (d+1)^{d+1}\sum_{(x_1, y_1), \dots ,(x_{k}, y_{k}) \in \Gamma_H \text{ distinct}}\PP(\forall i \in [k], \textbf{P}(x_i, y_i) = g_H(x_i, y_i))\\
    &\leq \sum_{k = 1}^{ d + 1} (d+1)^{d+1} (2q)^k q^{-k} \\
    &\leq (d + 1)^{d + 2} 2^{d + 1}. 
\end{align*}
\end{proof}

Let $V_H$ be the variety defined by  $\ell_H(x,y, x^2 - \alpha y^2) = 0$ and $\textbf{P}(x, y) = g_H(x, y)$ in $\FF_q^2$. Note that $|H\cap \bold{S}|=|V_H|$. On the other hand, Lemma~\ref{boundinglemma} says that either $|V_H| < t_0$ or $|V_H| \geq q/2$ when $q$ is sufficiently large.  
Now, combining this with the above claim and using Markov's inequality, we see that \begin{align*}
    \PP(|H \cap \bold{S}| \geq t_0) &= \PP(|H \cap \bold{S}| \geq q/2)\\ &= \PP(|H \cap \bold{S}|^{ d + 1} \geq (q/2)^{ d + 1})\\ &\leq \frac{\EE[|H \cap \bold{S}|^{ d + 1}]}{(q/2)^{ d + 1}}\\&\leq (d + 1)^{d + 2} 4^{ d + 1} q^{-1 - d}.
\end{align*}

Thus, applying a union bound over all $(1 + o(1))q^9$ many useful planes, we see that the probability that some useful plane has at least $t_0$ points in $\bold{S}$ is at most $2(d + 1)^{d+2}4^{d + 1} q^{8 - d}.$
\end{proof}

We also make the following observation. 
\begin{claim}\label{countonusefulplanes}
For at least $(1/4+o(1))q^9$ useful planes $H$, the curve $\Gamma_H$ from the proof of Lemma \ref{lemma-characterization} has size at least $q/2$. 
\end{claim}
\begin{proof}

Observe that there are $q^4$ vectors of the form $(x, y, x^2 - \alpha y^2, u, v)$ for $x, y, u, v \in \FF_q$. Each one lies on $(1 + o(1))q^6$ useful planes. Hence, $\sum_{H} |\Gamma_H|\geq (1+o(1))q^{10}$, where the sum is over all useful planes. Since there are (at most) $(1+o(1))q^9$ useful planes and for each useful plane $H$ we have $|\Gamma_H|\leq 2q$, the claim follows.   
\end{proof}

\begin{proof}[Proof of Lemma~\ref{goodsset}] Let $\bold{S} = \{( x, y, x^2 - \alpha y^2, z, z^2 + \textbf{P}(x, y)):x,y,z\in \mathbb{F}_q\}$ for some $\textbf{P}$ a uniform random polynomial of degree at most $d$.

By Claim~\ref{countonusefulplanes}, there are at least $(\frac{1}{4}+o(1))q^9$ useful planes $H$ for which $|\Gamma_H|\geq q/2$.  Fix such a plane $H$. Then the set of points $$W_H :=\{ (x, y, x^2 - \alpha y^2, g_{4, H}(x, y, x^2 - \alpha y^2), g_{5, H}(x, y, x^2 - \alpha y^2)): (x, y) \in \Gamma_H \} \subseteq H$$
has cardinality at least $q/2$. There are thus $\Omega(q^b)$ choices of a set of $b$ points on $W_H$. We have then $\Omega(q^{9 + b})$ many pairs consisting of a useful plane $H$ and a set of $b$ points on $W_H$. Note that each  point $(w_1, \dots ,w_5)\in W_H$ is in $\bold{S}$ if and only if $\bold{P}(w_1, w_2) = g_H(w_1, w_2)$. 
 Hence, the probability that a given set of $b$ points from $W_H$ 
lie on $\bold{S}$ 
is $q^{-b}$ by Lemma~\ref{conlon-nonzero}, since $d \geq b - 1$. Thus, the expected number of pairs consisting of a useful plane and $b$ points from $\bold{S}$ lying on it is $\Omega(q^9).$ Note that, by Lemma \ref{lem:no collinear triple}, no three points of $\bold{S}$ are collinear, so there is a unique plane containing any $b$ coplanar points of $\bold{S}$. Hence, the expected number of $b$-element point sets in $\bold{S}$ which are coplanar on a useful plane is $\Omega(q^9)$.

Let $X$ be the random variable counting the number of $b$-element point sets in $\bold{S}$ which are coplanar on a useful plane. We have just proved that $\mathbb{E}[X]=\Omega(q^9)$. Let $A$ be the event that every useful plane contains fewer than $t_0$ points of $\bold{S}$. By Lemma \ref{lemma-characterization}, the probability of $A$ failing is at most $O(q^{8 - d})$. Note that since no three points of $\bold{S}$ are collinear, every plane contains at most $2q$ points of $\bold{S}$, and therefore the number of coplanar $b$-element sets in $\bold{S}$ is at most $O(q^{9+b})$.

Hence,
$$\mathbb{E}[X \bold{1}_A]=\mathbb{E}[X]-\mathbb{E}[X \bold{1}_{A^c}]\geq \mathbb{E}[X]-O(q^{9+b})\mathbb{P}(A^c)\geq \Omega(q^9)-O(q^{9+b}q^{8-d})\geq \Omega(q^9),$$
since $d=b+9$. Thus, there exists an outcome $S$ of $\bold{S}$ such that the event $A$ holds and the number of $b$-element point sets in $\bold{S}$ which are coplanar on a useful plane is $\Omega(q^9)$.

By Lemma~\ref{lemma-characterization}, there are $q^2$ risky planes which intersect $S$ in exactly $q$ points, and thus there are at least $q^2$ dangerous planes. On the other hand, we claim that no other plane contains at least $t_0$ points of $S$ (so in particular there are exactly $q^2$ dangerous planes). Indeed, by Lemma \ref{lemma-characterization}, each tolerable plane contains at most $4$ points of $S$, and since the event $A$ holds, each useful plane contains fewer than $t_0$ points of $S$.

Thus, $S$ satisfies the conclusion of the lemma. 
\end{proof}

\section{Infinitely many realizable exponents} \label{sec:realizable exponents}

In this section, we prove Theorem \ref{thm:realizable exponents}. The key result is as follows.

\begin{proposition} \label{prop:realizable upper bound}
    Let $F$ be a graph and let $s\geq v(F) + e(F)$. Then $\ex(n,F,K_{s,t})=O(n^{v(F)-e(F)/s})$.
\end{proposition}

    \begin{proof}

        Let $G$ be a $K_{s, t}$-free graph on $n$ vertices. Set $p = C_0n^{-\frac{1}{s}}$ for a sufficiently large constant $C_0$ that depends only on $s$ and $t$. 

        We say that a set $S\subset V(G)$ of size $i$ is \emph{heavy} if $|N_G(S)| > p^in$ and \emph{light} otherwise. (Here and below, $N_G(S)$ denotes the common neighbourhood of the set $S$ in the graph $G$. We use the convention that the common neighbourhood of the empty set is $V(G)$, so the empty set is light.) We say that a set is \emph{admissible} if it is heavy and every proper subset is light. 

        We will count the number of copies of $F$ according to whether they contain an admissible set of vertices or not. We bound those that do contain an admissible set by the following claim.

        \begin{claim} \label{claim:few admissible}
            For every $1\leq i \leq v(F)$, the number of admissible sets of size $i$ is $O(n^{i - 1 + \frac{v(F)}{s}})$. 
        \end{claim}
        \begin{proof}
            Given an admissible set $S$ of size $i$, let $T\subset S$ be an arbitrary subset of $S$ of size $i-1$. Since $S$ is admissible, $T$ is light, so $|N_G(T)|\leq p^{i-1}n$. On the other hand, since $S$ is heavy, we have $|N_G(S)|>p^i n$. The claim follows if we can show that there are at most $O(n^{\frac{v(F)}{s}})$ vertices $v\in V(G)\setminus T$ such that $|N_G(v) \cap N_G(T)| > p^i n$ as then, for a given $T$, there are at most $O(n^{\frac{v(F)}{s}})$ ways to choose the unique vertex in $S\setminus T$. Let $X$ be the set of vertices $v\in V(G)\setminus T$ such that $|N_G(v) \cap N_G(T)| > p^i n$. Let $L$ be the bipartite graph with parts $X$ and $N_G(T)$ with an edge between a pair if there is such an edge in $G$. Since $G$ is $K_{s, t}$-free, $L$ does not contain $K_{s-i+1,t}$ with the part of size $s-i+1$ in $X$. Indeed, any such copy of $K_{s-i+1,t}$ would yield a copy of $K_{s,t}$ in $G$ by adding $T$ to it.

            On the one hand, $e(L) \geq p^i n|X|$ since each $v\in X$ has degree at least $p^i n$ in $L$. On the other hand, by the K\H ov\'ari--S\'os--Tur\'an theorem \cite{KST54}, we have that $e(L) \leq C(|X||N_G(T)|^{1-1/(s-i+1)}+|N_G(T)|)$ for some constant $C$ depending only on $s$ and $t$. It follows that either $e(L)\leq 2C|X||N_G(T)|^{1-1/(s-i+1)}$ or $e(L)\leq 2C|N_G(T)|$. In the former case, we obtain $p^in\leq 2C|N_G(T)|^{1-1/(s-i+1)}\leq 2C(p^{i-1}n)^{1-1/(s-i+1)}$, where we used that $|N_G(T)|\leq p^{i-1}n$. A simple computation shows that this is not possible if $C_0$ is sufficiently large. In the latter case, we have $p^i n|X|\leq 2C|N_G(T)|\leq 2Cn$. Hence, $|X|\leq 2Cp^{-i}\leq O(n^{i/s})\leq O(n^{v(F)/s})$. 
        \end{proof}

        The following claim bounds the number of copies that do not contain an admissible set of vertices.

         \begin{claim} \label{claim:no admissible}
             The number of copies of $F$ not containing an admissible set of vertices is at most $n^{v(F)}p^{e(F)}$.
         \end{claim}

        \begin{proof}
        Note that if in a copy of $F$, no set of vertices is admissible, then every subset of its vertices is light. We bound the number of such copies as follows. Let $v_1,\dots,v_k$ be an arbitrary ordering of the vertices of $F$. We will prove by induction on $i$ that the number of ways to embed $F_i:=F[\{v_1,\dots,v_i\}]$ in $G$ such that every subset of the vertex set is light is at most $n^{v(F_i)}p^{e(F_i)}$. This then completes the proof of the claim by taking $i=k$.

        The base case ($i=0$) of the induction is clear. For the induction step, it suffices to prove that for any embedding $\varphi$ of $F_i$ into $G$ such all vertex subsets of $\varphi(V(F_i))$ are light, the number of choices for the image of $v_{i+1}$ which extend $\varphi$ to an embedding of $F_{i+1}$ is at most $p^{e(F_{i+1})-e(F_i)}n$. Let $Z=N_{F_{i+1}}(v_{i+1})$ and let $S = \varphi(Z)\subset V(G)$ be the image of $Z$. Note that $v_{i+1}$ needs to be mapped to a vertex in $N_G(S)$ in order to obtain an embedding of $F_{i+1}$. However, as $S$ is light, there are at most $p^{|S|}n$ such vertices. Noting that $|S|=|Z|=d_{F_{i+1}}(v_{i+1})=e(F_{i+1})-e(F_i)$, it follows that indeed the number of choices for the image of $v_{i+1}$ which extend $\varphi$ to an embedding of $F_{i+1}$ is at most $p^{e(F_{i+1})-e(F_i)}n$.
        \end{proof}

        By Claim \ref{claim:few admissible}, the number of copies of $F$ containing an admissible set of any size is $O(n^{v(F) -1 + \frac{v(F)}{s}})= O(n^{v(F) - \frac{e(F)}{s}})$ by our choice of $s$. 
        On the other hand, by Claim \ref{claim:no admissible}, the number of copies of $F$ not containing an admissible set is at most $n^{v(F)}p^{e(F)}=O(n^{v(F)-e(F)/s})$. Thus, $G$ contains at most $O(n^{v(F) - e(F)/s})$ copies of $F$, completing the proof. 
    \end{proof}

It is now easy to prove Theorem \ref{thm:realizable exponents}.

\begin{proof}[Proof of Theorem \ref{thm:realizable exponents}]
    By Propositions \ref{prop:realizable lower bound} and \ref{prop:realizable upper bound}, the number $v(F)-e(F)/s$ is a realizable exponent for $F$ for all $s\geq v(F) + e(F)$ (with $H=K_{s,t}$ for a sufficiently large~$t$). Since $e(F)>0$, this gives infinitely many realizable exponents.
\end{proof}

\section{Concluding remarks}

Given Theorem \ref{thm:k3t-free}, it is natural to make the following conjecture.

\begin{conjecture}
    For any $s<a\leq b$, there exists some $t_0$ such that for all $t\geq t_0$, we have
    $$\ex(n,K_{a,b},K_{s,t})=\Theta_{s,t,a,b}(n^s).$$
\end{conjecture}

The natural extension of our approach would require us to construct a set of roughly $q^s$ points $\mathcal{P}$ and a set of roughly $q^s$ one-codimensional hyperplanes $\mathcal{H}$ in $\mathbb{F}_q^{2s-1}$ such that the incidence graph is $K_{s,t}$-free but it contains $\Omega(q^{s^2})$ copies of $K_{a,b}$. When $s\geq 4$, we would like to avoid having four coplanar points in $\mathcal{P}$ as such points are likely to be contained in a copy of $K_{s,t}$. However, it is not hard to show that any point set $\mathcal{P}\subset \mathbb{F}_q^{2s-1}$ with no four coplanar elements has size at most $O(q^{s-1})$. Indeed, any such $\mathcal{P}$ would in particular have no three points on a line. Now, for any $x,y\in \mathcal{P}$, consider the line $\ell_{x,y}$ between $x$ and $y$. Since $\mathcal{P}$ does not contain four coplanar points, it follows that the sets $\ell_{x,y}\setminus \{x,y\}$ are pairwise disjoint over all pairs $x,y$. This implies that $\binom{|\mathcal{P}|}{2}(q-2)\leq q^{2s-1}$, from which $|\mathcal{P}|=O(q^{s-1})$ follows.

\medskip

\paragraph{Acknowledgements.} We are very grateful to Sam Spiro for useful discussions on the topic of this paper. We would like to also acknowledge the support of NSF through Grant DMS-2321249, for the $30$th edition of the Atlanta Lecture Series, during which this research was initiated.  

\paragraph{AI declaration.} While the authors found all the main ideas of the proofs, ChatGPT was used to work out some of the technical details.

\bibliography{ref}

\bibliographystyle{abbrv}

\end{document}